\documentclass{qam-l}

\usepackage[T1]{fontenc}
\usepackage{verbatim}
\usepackage{graphicx}
\usepackage{amssymb}

\usepackage{color}
\definecolor{MyLinkColor}{rgb}{0,0,0.4}

\newcommand{\sign}{\mathop{\rm sign}\nolimits}

\newcommand{\0}{\Omega}

\newcommand{\p}{\partial}

\newcommand{\R}{\mathbb{R}}

\newcommand{\N}{\mathbb{N}}
\newcommand{\Z}{\mathbb{Z}}

\newtheorem{thm}{Theorem}[section]
\newtheorem{prop}[thm]{Proposition}
\newtheorem{lemma}[thm]{Lemma}

\theoremstyle{remark}

\numberwithin{equation}{section}

\title[On particle motion in geophysical deep water waves]
{On the particle motion in geophysical deep water waves traveling over uniform currents}


\author[Anca--Voichita Matioc]{Anca--Voichita Matioc}
\address{Faculty of Mathematics,  University of Vienna, Nordbergstrasse 15, 1090 Vienna, Austria.}
\email{anca.matioc@univie.ac.at}

\begin{document}

\begin{abstract}
We describe a family of exact Gerstner type solutions for the geophysical equatorial deep water wave problem in the $f-$plane approximation. 
These Gerstner type waves are two-dimensional and  travel with constant speed over a uniform horizontal current.
The particle paths in the presence and absence of the Coriolis force are also analyzed in dependence of the current strength. 
\end{abstract}

\maketitle

\section{Introduction}
Analyzing the motion of an object situated  in a  reference frame rotating with uniform angular velocity one has to take into consideration the Coriolis force acting upon the object.
In particular, when studying geophysical water waves traveling over an inviscid  fluid at the Earth's surface, due to the rotation of the Earth around its axis, the
 Coriolis  force  influences the motion of the fluid particles, cf. \cite{GalStR07, Pe79}
 adding also additional terms in the Euler equations. 

There exists an explicit solution for gravity deep water waves  without Coriolis effect, which was found first by Gerstner \cite{GE09} and later on by  Rankine \cite{Ra63}.
Its features have been analyzed  in \cite{Co01, He08} and more recently in \cite{AM12xx} where the effects caused by the Earth's rotation are also considered.
Gerstner's solution    is in fact the unique solution for the geophysical  deep water wave  that has no stagnation points and with the pressure constant along the streamlines, cf. \cite{Ka04, MM12x}. 
It is given by  describing the evolution of each individual fluid particle in the flow: all fluid particle move  on  circles,  the  radii of the circles decreasing with the depth.  
Moreover, the flow is rotational and the vorticity decays with depth.

That for Gerstner's wave  fluid particles move on circles is in agreement with the  
 classical description of the particle paths  within the framework of linear water wave theory \cite{Joh97,  Li78,  St92}: all water particles
trace a circular orbit, the diameter of which decreases with depth.
However, it was recently shown within linear theory \cite{CoEhVi08, AM12} that for irrotational periodic water waves the particle paths are not closed. 
Even within the linear water wave theory, the ordinary
differential equations system describing the motion of the fluid particles is nevertheless nonlinear and explicit solutions of
this system are not available.
However, qualitative features of the underlying flow have been obtained in a nonlinear setting in \cite{Con11,  Co-Es04_1, He06,  BM12xx, Mat12,  To96}.
The particle trajectories and other properties for the flow beneath water waves of finite depth have been discussed in \cite{ Co06, CoEhWa07, CoEs04_b,  AC11, CoVi08, DH07, ioK08, MaA10,okamoto-shoji-01}, to mention just some of the contributions.
  
It is worth mentioning that there exists an explicit solution describing the propagation of edge-waves along a sloping beach which is obtained in \cite{C01a} by adapting Gerstner's solution. 
Gerstner's idea may be used to construct an explicit solution for non-homogeneous deep water waves and for edge-waves  along a sloping beach cf. \cite{Ra11}.    
More recently, in \cite{AM12a} the author gave an explicit solution describing  geophysical equatorial edge waves propagating over a plane sloping beach with the shoreline parallel to the Equator.

For waves localized close to the Equator, the mathematical formulation of the geophysical water wave problem  is the so-called $f-$plane approximation, the Coriolis parameter being constant in this case.
For the physical relevance of our setting we refer to  the discussion in \cite{AC12}.
In this paper we analyze a family of explicit Gerstner type solutions, for the  geophysical deep water wave problem  in the $f-$plane approximation, which travel over an uniform horizontal current.
These solutions are presented  in a Lagrangian framework by describing the path of each individual fluid particle.
They   appeared before in a implicit form in \cite{MC78} in the context of gravitational and geostrophic billows.
In the absence of an underlying current, the solutions presented herein coincide with those studied in \cite{Co01, He08, AM12xx}.
In this case the particle trajectories are circles. 
In this paper we investigate the influence of the   underlying current  on the  wave and the  flow beneath it.
Particularly, we prove that they possess the following features:
\begin{itemize}
 \item the waves propagate along the Equator and are geometrically reflections of curtate cycloids or of cycloids with cusps with respect to a horizontal line. 
They have  wavelength $\lambda=2\pi/k$,  $k\in\N.$ 
Their form is independent of the strength $U$ of the uniform current, but the wave speed $c$ is $U$-dependent;
\item depending on their position in the flow and the value of $U,$ when $U\neq0,$ the fluid particles are looping or undulate in the direction of the current (also when the current and the wave move in opposite directions).
Their horizontal drift though is determined only by $U$ (and  $c$ for geophysical waves), and is uniform with respect to the position of the particles within the wave;   
\item the flows are rotational and the absolute value of the vorticity decreases with depth. 
Its value is proportional to $U-c$ for geophysical waves, whereas in the absence of Coriolis effects it is only the sign of the vorticity which changes with $U-c;$   
\item the influence of the underlying current on the flow beneath the wave depends in the geophysical context on the direction of propagation of the wave along the Equator. 
There is a limitation on the value of $U$ above which our functions  stop to describe solutions of the deep water wave problem.
For the limit value of $U$ the solutions describe waves propagating from east to west with speed $c=U$ and consisting only of stagnation points. 
\end{itemize}

The outline of the paper is the following: we present  in Section 2 the governing equations in the $f-$plane approximation  and, using a Lagrangian description, we introduce 
the family of Gerstner type solutions traveling over uniform horizontal currents. 
In Section \ref{Sec:2} we prove that these solutions verify the equations of motion and finally, in Section \ref{Sec:3},  we analyze and illustrate the paths of the particles in dependence of the current strength,
 in the absence of the Coriolis force and for  geophysical waves, respectively. 

\section{The mathematical model and the Gerstner type solutions}\label{Sec:1}
In a rotating frame with the origin at a point on Earth's surface, with the $X-$axis chosen horizontally due east, the $Y-$axis horizontally due north and the $Z-$axis  upward, 
we consider waves traveling over a fluid of infinite depth, the wave surface being the graph  $Z=\eta(t,X,Y)$.
For a fluid layer localized near the Equator, the governing equations in the $f-$plane approximation are the Euler equations
\begin{subequations}\label{Pb1}
\begin{equation}\label{Euler}
\left\{
\begin{array}{rllllll}
u_t+uu_X+vu_Y+wu_Z+2\omega w&=&-P_X/\rho \\
v_t+uv_X+vv_Y+wv_Z&=& -P_Y/\rho,\\
w_t+uw_X+vw_Y+ww_Z-2\omega u&=&-P_Z/\rho -g,
\end{array}
\right.
\end{equation}
and the equation of mass conservation 
\begin{equation}\label{maco}
u_X+v_Y+w_Z=0.
\end{equation}
Here $t$ represents time, $(u,v,w)$ is the fluid velocity, $\omega=73\cdot 10^{-6} rad/s$ is the (constant) rotational speed of the Earth\footnote[1]{Taken to be a perfect sphere of radius $6371 km$.} round the polar axis
towards east, $\rho$ is the (constant) density of the water, $g=9,8 m/s^2$ is the  gravitational constant  at the Earth's surface, and $P$ is the pressure.  
The free surface decouples the motion of the water from that of the air  by the dynamic boundary condition
\begin{equation}\label{Dbc}
P=P_0 \quad \quad \text{on} \quad Z=\eta(t,X,Y),
\end{equation}
where $P_0$ is the constant atmospheric pressure.
Since the same particles always form the free surface, we also have the kinematic
boundary condition
\begin{equation}\label{kbc}
w=\eta_t+u\eta_X \quad \quad \text{on} \quad Z=\eta(t,X,Y).
\end{equation}
Finally, assuming that at great depths the flow is almost horizontal we impose that
\begin{equation}\label{B2}
(u,w)\to (U,0) \quad \quad \text{for $ Z \to-\infty$ uniformly in $(t,X)$.}
\end{equation}
\end{subequations}

In this paper we study  traveling wave solutions of \eqref{Pb1}, that is solutions with the velocity field, the pressure, and the 
free surface exhibiting a $(t,X)-$dependence of the form $(X-ct),$ where $|c|$ is the speed of the wave surface.  
If $c>0,$ then the wave moves from west to east, while if $c<0$ it moves in the opposite direction, with constant speed $|c|.$
Moreover, we seek two-dimensional flows, independent upon the $Y-$coordinate and with $v \equiv 0$ throughout the flow, that is solutions of 
\begin{equation}\label{Pb}
\left\{
\begin{array}{rrlllll}
u_t+uu_X+wu_Z+2\omega w&=&-P_X/\rho, &\text{in $ Z<\eta(X-ct),$} \\
w_t+uw_X+ww_Z-2\omega u&=&-P_Z/\rho -g,&\text{in $ Z<\eta(X-ct),$}\\
u_X+w_Z&=&0&\text{in $ Z<\eta(X-ct),$}\\
P&=&P_0&\text{on $ Z=\eta(X-ct),$}\\
w&=&\eta_t+u\eta_X&\text{on $ Z=\eta(X-ct),$}\\
(u,w)&\to& (U,0) & \text{for $ Z \to-\infty$ uniformly in $(t,X)$.}
\end{array}
\right.
\end{equation}
Note that formally putting   $\omega=0$ in  \eqref{Pb}, that is neglecting Coriolis force, 
the system \eqref{Pb} reduces to the classical problem for two-dimensional deep water waves.

We describe now a family of exact solutions to the water wave problem \eqref{Pb}.
We will allow for $\omega=0$, so that in this case our solutions describe deep water waves  propagating in the absence of Coriolis effects. 
These special, Gerstner type, solutions of the problem \eqref{Pb} are given in a Lagrangian framework  by describing the paths of all fluid particles within the wave
\begin{equation}\label{Lag}
\left\{ 
\begin{array}{llll}
&\displaystyle X(t,a,b):=a+\left(c-m\right)t-\frac{e^{kb}}{k}\sin(k(a-mt)),\\[1ex]
&\displaystyle Z(t,a,b):= b+\frac{e^{kb}}{k}\cos(k(a-mt))
\end{array}
\right.
\end{equation}
 for all  $a\in\R$, $b\leq b_0,$ and $t\geq0.$  
Hereby, the wave number $k$ is assumed positive and 
\begin{equation}\label{regime}
b_0\leq 0. 
\end{equation} 
Each particle within the fluid is uniquely determined by a pair  $(a,b)$ with $b\leq b_0,$
the curve $(X(t, \cdot, b_0), Z(t, \cdot, b_0))$ being  the profile of the wave. 
The path of each individual particle is  described by the mapping  $(X( \cdot, a, b_0), Z( \cdot, a, b_0))$.
As we shall see in Section \ref{Sec:2}, the equations \eqref{Lag} describe deep water waves traveling over a uniform horizontal current $(U,0),$ with $U=c-m,$
  the wave speed $c,$ and the parameter $m$ satisfying certain constraints. 
Depending on $c$ and $m$, the current may be favorable, that is the wave and the current  move in the same direction,
 or  adverse, if the current moves  in the  opposite direction to that of wave propagation.

\section{Analysis of the Gerstner's  family of solutions}\label{Sec:2}
To keep our notation short we introduce the parameter set $\Sigma:=\R\times(-\infty,b_0).$
First, we prove that at any time $t\geq0$ the set $\{(X(t,a,b), Z(t,a,b))\,:\, (a,b)\in\Sigma\}$ describe an infinite 
fluid layer bounded from  above by the free surface  $(X(t,\cdot,b_0), Z(t,\cdot,b_0))$, which is a graph when assuming $b_0\leq0.$
\begin{lemma}\label{L:1}
 Given  $t\geq 0,$ the map $\Phi(t):=(X(t), Z(t))$ defines a real-analytic diffeomorphism from $\Sigma$ onto  its image $\0(t).$
Moreover, there exists a function $\eta:\R\to\R,$ which is periodic of minimal period $2\pi/k$ such that 
\[
\0(t)=\{(X,Z)\,:\, \text{$X\in\R$  and $Z<\eta(t,X):=\eta(X-ct)$}\},
\] 
with $\Phi(t):\partial \Sigma \to \partial \Omega$ being a bijective map.
\end{lemma}
\begin{proof} Let us note that for each $b\leq b_0$ and $t\geq0,$ the mapping $\R\ni a\mapsto X(t,a,b)\in\R$ is a diffeomorphism. 
Let $X^{-1}(t,\cdot,b):\R\to\R$ denote its inverse.
The image of $\Phi(t)$ is then the set
\[\0(t)=\{(X, Z(t, X^{-1}(t,X,b),b))\,:\, (X,b)\in\R\times(-\infty,b_0)\}.\]
Since 
\begin{align*}
\frac{\p}{\p b}\left(Z(t, X^{-1}(t,X,b),b))\right)&=\frac{\det \p\Phi}{X_a}(t, X^{-1}(t,X,b),b))\\
&=\frac{1-e^{2kb}}{1-e^{kb}\cos(k(a-mt))}\big|_{a=X^{-1}(t,X,b)}>0
\end{align*}
for all $t\geq0$ and $(X,b)\in\R\times(-\infty,b_0)$, and $Z(t, X^{-1}(t,X,b),b)\to_{b\to-\infty}-\infty$ we conclude that $\Phi(t)$
maps $\Sigma$ bijectively onto
\[\0(t)=\{(X, Z)\,:\, Z<Z(t, X^{-1}(t,X,b_0),b_0))\}.\]

Thus, at time $t\geq0,$ the wave surface is the graph $Z=\eta(t,X).$
Letting $\eta:=\eta(0,\cdot),$ to show that $\eta(t,X)=\eta(X-ct)$ reduces to prove that
\[
X^{-1}(t,X,b_0)-mt=X^{-1}(0,X-ct,b_0)\qquad\text{for all $t\geq0$ and $X\in\R$.}
\]
The latter identity can be easily obtained from \eqref{Lag}.
Similarly, that $\eta$ has minimal period $T$ reduces to finding a constant $T>0$ such that 
 \[X^{-1}(0,X+T,b_0)=X^{-1}(0,X,b_0)+\frac{2\pi}{k},\]
and it follows readily from \eqref{Lag} that the minimal period is $T=2\pi/k.$
\end{proof}

Observe  that the constants $c$ and $m$ do not interfere in the definition of $\eta,$ so that the shape of the wave surface is independent of these constants.
Similarly as in the case studied in \cite{Co01, He08, AM12xx}, the wave surface is the reflection of a curtate cycloids or of a cycloid with cusps with respect to a horizontal line (see also the discussion at the beginning of Section \ref{Sec:3}).

We emphasize that by Lemma \ref{L:1}, the constant $c$ is the only candidate for the wave speed. 
As we shall see in Lemmas \ref{L:2a}, \ref{L:2b}, and \ref{L:3} below, $c$ is indeed the wave speed constant.
To prove this we have to find first a pressure $P$ so that all equations of \eqref{Pb} are satisfied and show also that $(u,v,P)$ has a $(t,X)$ dependence of the form $(X-ct).$  

To proceed, since the paths of the particles are described by \eqref{Lag}, we see that the velocity  vector field within the fluid layer is given at any time $t\geq0$
 by the relations
\begin{equation}\label{velocity}
\left\{
\begin{array}{llll}
 u(t,X,Y)=X_t(t,a,b)|_{(a,b)=\Phi(t)^{-1}(X,Y)},\\
 v(t,X,Y)=Z_t(t,a,b)|_{(a,b)=\Phi(t)^{-1}(X,Y)}
\end{array}
\right.\qquad\text{for $(X,Y)\in\overline{\0(t)}.$}
\end{equation}

In the next lemma we show that, under certain restrictions on $c$ and $m$, there exists a pressure $P$, which depends on $(X, Z)$  only through  the parameter $b$, 
such that the Euler equations are satisfied.
We will differentiate now between flows with Coriolis effects and flows where $\omega=0$. 
As $(X, Z)$ are given explicitly by using the parameters $(a,b)$, it is more natural to work with the latter variables  when determining the pressure of the flow (also the vorticity, cf. Lemma \ref{L:5}).

\begin{lemma}\label{L:2a} If $\omega=0,$ we assume that 
\begin{equation}\label{C:1}
m^2= \frac{g}{k}.
\end{equation}
  Then, letting 
\begin{equation}\label{pre}
 P:=\frac{\rho m(km+2\omega)}{2k}(e^{2kb}-e^{2kb_0})+\left[2\omega\rho (c-m)-g\rho\right](b-b_0)+P_0,\qquad b\leq b_0,
\end{equation}
the first two equations of \eqref{Pb} are satisfied for all $t\geq0 $ and all $c\in\R.$
\end{lemma}
We emphasize that when Coriolis effects  are neglected, Lemma \ref{L:2a} ensures that  there are no restrictions on the wave speed $c.$
Note that  we recover from \eqref{Lag} and \eqref{C:1}, when $m=c=\sqrt{g/k}$,  Gerstner's solution studied in \cite{Co01,He08}. 

 \begin{lemma}\label{L:2b} If $\omega>0,$ the speed of the wave is given by
\begin{equation}\label{speed}
 c:=\frac{g-km^2}{2\omega}
\end{equation}
 and, letting $P$ be the function defined by \eqref{pre} (with $c$ given by \eqref{speed}),
the first two equations of \eqref{Pb} are satisfied for all $t\geq0 $.
\end{lemma}

Thus, in the context of geophysical waves,  the wave speed is determined by $m$ and $k$, but the parameter $m$ is free now.
While the geophysical waves \eqref{Lag} may travel from east to west at any speed, as $c$ can be chosen to be any negative number, cf. \ref{speed}, the speed of propagation for waves moving  from west to east
is bounded by $g/(2\omega).$     
In the limit case $c=g/(2\omega),$ all particles within the fluid are stagnation points as they move horizontally with the same speed as the wave, cf. Proposition \ref{P:1}.

We shall prove both Lemmas \ref{L:2a} and \ref{L:2b} simultaneously.
\begin{proof}[Proof of Lemmas \ref{L:2a} and \ref{L:2b}]
Similar arguments to those in \cite{AM12xx} show that we are left to determine a function $P $ which has the property that it is constant on the wave surface and
\begin{align}
 P_a&=-\rho(X_{tt}+2\omega Z_t)X_a-\rho(Z_{tt} -2 \omega X_t+g) Z_a, \label{p1}\\
 P_b&= -\rho(X_{tt}+2 \omega Z_t)X_b-\rho(Z_{tt}-2 \omega  X_t+g) Z_b \label{p2}
\end{align}
for all $t\geq0$ and $(a,b)\in\Sigma.$
Using \eqref{Lag}, the relation   \eqref{p1} is equivalent to 
\[
P_a=-\rho e^{kb}\sin(k(a-mt))\left(km^2+2\omega c-g\right)
\] 
for all $t\geq0$ and $(a,b)\in\Sigma.$
Since the pressure should be constant $P=P_0$ at the wave surface $b=b_0,$ we have to impose that 
\begin{align*}
 km^2+2\omega c-g=0,
\end{align*}
and we arrive at \eqref{C:1} or \eqref{speed}.
With this choice, we see that $P_a\equiv0$ in $\Sigma.$
Concerning \eqref{p2}, we find by direct  computation  that
\[
P_b=\rho m(km+2\omega)e^{2kb}+2\omega\rho (c-m)-g\rho,
\]
which leads us to relation \eqref{pre}.
\end{proof}

We prove next that $u, v$, and $P$ exhibit a $(t,X)$ dependence of the form $(X-ct)$ and that $(u,w)$ satisfies the far-field condition \eqref{B2}.
The constant $U=c-m$ will be identified as  being the horizontal velocity of the uniform background current.
 \begin{lemma}\label{L:3} We have that
 \[ (u,w,P)(t,X,Z)=(u,w,P)(0,X-ct,Z)\]
for all   $(X,Z)\in\0(t) $, $t\geq0$,  and
\[
(u,w)\to(c-m,0)\qquad\text{for $Z\to-\infty$ uniformly in $(t,X).$}
\]
\end{lemma}
\begin{proof} Given $(X,Z)\in\0(t),$ we infer from Lemma \ref{L:1} that $(X-ct,Z)\in\0(0),$
so that the relation which is to be proved is well-defined.
 Setting $\Phi(t)^{-1}=(a(t),b(t)),$ it  follows readily from \eqref{Lag} and \eqref{pre} 
that it suffices to establish that
 \[a(t,X,Z)-mt=a(0,X-ct,Z)\qquad\text{and}\qquad b(t,X,Z)=b(0,X-ct,Z)\]
for all $t\geq0$ and $(X,Z)\in\0(t).$
However, the latter identities are equivalent to 
\[X(t,a,b)-ct=X(0,a-mt,b) \qquad\text{and}\qquad Z(t,a,b)=Z(0,a-mt,b) \]
for all $t\geq0$ and $(a,b)\in\Sigma,$ and these properties follows easily from \eqref{Lag}.

From \eqref{Lag} and \eqref{velocity} we obtain the far-field condition, and the proof is completed.
 \end{proof}

We show now that the flow is incompressible.
\begin{lemma}\label{L:4} The flow is incompressible, that is for all $t\geq0$ we have that
\[u_X+w_Z=0\text{\qquad in $\0(t).$}\] 
\end{lemma}
\begin{proof}
 Using the same notation $\Phi(t)^{-1}=(a(t),b(t))$ as in the proof of Lemma \ref{L:3},
we obtain by the chain rule that
\begin{align*}
u_X+w_Z&=X_{ta}a_X+X_{tb}b_X+Z_{ta}a_Z+Z_{tb}b_Z=\frac{X_{ta}Z_b-X_{tb}Z_a-Z_{ta}Z_b+Z_{tb}X_a}{\det \p\Phi(t)}=0,
\end{align*}
cf. \eqref{Lag}.
\end{proof}

Next, we show that if $m\neq0$, then the flow is rotational and that the vorticity decreases rapidly with the depth.
For $m=0$ the particles move horizontally  with the wave speed, meaning that the flow is irrotational. 
\begin{lemma}\label{L:5}
 The vorticity of the flow is given by the relation
\begin{align}\label{vort}
 \gamma=-\frac{2kme^{2kb}}{1-e^{2kb}} \qquad\text{for $b\leq b_0$}.
\end{align}
\end{lemma}
\begin{proof}
It follows from \eqref{Lag} and \eqref{velocity} by direct computation that
\begin{align*}
\gamma=u_Z-w_X&=\frac{X_{tb}X_a-X_{ta}X_b-Z_{ta}Z_b+Z_{tb}Z_a}{\det \p\Phi(t)}=-\frac{2kme^{2kb}}{1-e^{2kb}}.
\end{align*}
\end{proof}

We notice from relation \eqref{vort} that the vorticity function depends only of the $sign$ of $m$ when $m=0$, 
whereas  for geophysical waves it depends upon the difference $U-c$. 

Finally, we are left to check the second last relation of \eqref{Pb}.  

\begin{lemma}\label{L:8}
The kinematic  condition 
 \[w=\eta_t+u\eta_X \quad \text{on $ Z=\eta(t,X)$}\]
is satisfied for all $t\geq0.$
\end{lemma}
\begin{proof}
 Using Lemma \ref{L:1}, we need to prove that  $w=(u-c)\eta_X$ on $Z=\eta(t,X)$.
By the definition of $\eta$ we have
\[
\eta_X(X(t,a,b_0))=\frac{Z_a}{X_a}(t,a,b_0)=-\frac{e^{kb_0}\sin(k(a-mt))}{1-e^{kb_0}\cos(k(a-mt))}.
\]
Using \eqref{Lag} and the previous relation, we obtain that
\begin{align*}
 w -(u -c)\eta_X =&me^{kb_0}\sin(ka-kct)\\
&+m(e^{kb_0}\cos(k(a-mt))-1)\frac{e^{kb_0}\sin(k(a-mt))}{1-e^{kb_0}\cos(k(a-mt))}\\
=&0,
\end{align*}
 and therefore the kinematic surface condition is satisfied.
\end{proof}

\section{Particle paths in geophysical waves}\label{Sec:3}
As mentioned in Section \ref{Sec:2}, the wave profile corresponding to the solutions  \eqref{Lag} is independent of the  magnitude of the underlying current  $U=c-m$ and  the value of $m$.
It is the motion of the particles and the speed of the wave which are influenced by these constants.
While for waves without Coriolis effects there is a symmetry with respect to $\sign(m)$,  the particle description for geophysical waves depends non-trivially on both constants $U$ and  $m$. 
For example if there is no underlying current $U=0$ there are two 
solutions in the family \eqref{Lag}: one moving from west to east with wave speed $c=(-\omega+\sqrt{\omega^2+kg})/k,$   and one moving in the opposite direction at larger speed  $c=(\omega+\sqrt{\omega^2+kg})/k.$\footnote{
Note that in the absence of Coriolis effects $\omega=0$, the speed of the two solutions with $U=0$ is the same $c=\sqrt{g/k}.$ 
In fact, in this case, it is the same solution but seen from two different reference frames.
However, when $\omega>0,$ not only that the wave speed and the direction of propagation are different, but also  the angular velocity of the  particles, which move in both cases on circles, is different.} 

We start by examining  the drift experienced by the water particle, that is the horizontal distance moved by the particle between its positions on two consecutive crest lines.
If $m=0$, then particles on a crest line remain always there, and the notion of drift has no meaning. 
Assume now $m\neq0$.
Clearly, the particle parametrized by $(a,b)=(0,0)$ is located at time $t=0$ on a crest line.
 We infer then from \eqref{Lag} that the first time $t>0$ when the particle is located again on a crest line is $t=2\pi/(k|m|)$.  
The  drift of this particle is
\[
d=X(2\pi/k|m|),0,0)-X(0,0,0)=\frac{2\pi U}{k|m|}.
\]
If $(0,b)$ corresponds to  another particle on the crest line $X=0 $ (when $t=0$) we infer from \eqref{Lag} that 
  \[X(2\pi/k|m|),0,b)-X(0,0,b)=\frac{2\pi U}{k|m|},\]
meaning  that $t=2\pi/(k|m|)$ is the first positive time when the particle is located again on a crest line.
Thus, we may conclude the following result.    
\begin{prop}  \label{FD}
Assume that $m\neq0.$    The signed horizontal drift of a particle  within a period is
\begin{equation}\label{eq:displ1}
d=\frac{2\pi U}{km \sign(m)}.
\end{equation}
\end{prop}

For Stokes waves the drift of   particles depends on their  vertical position in the fluid and it decreases with depth, cf. \cite{CoEhVi08, CoSt10, He06}.
Moreover, in the presence of adverse currents it is shown in \cite{CoSt10} that some of the water particles may move in the direction of the current and other ones in the opposite one. 
For our Gerstner type solutions though, the drift of the particles is constant in the flow and all particles move (loop or undulate) in the same direction as the current.

In order to describe the trajectory traced  by a water particle,  which corresponds in the parameter space $\Sigma$ to the pair $(a,b)$, it is suitable to re-parametrize this curve by rescaling the time as follows 
\begin{equation}\label{var}
\tau:=\sign(m)k(mt-a)\qquad\text{for $m\neq0$}.
\end{equation} 
Note that this change of variables is orientation-preserving.
We find that the curve
\begin{equation}\label{param}
(x(\tau),z(\tau)):=\left(\frac{ac}{m}+\frac{U}{km\sign(m)}\tau+\frac{e^{kb}}{k}\sin(\sign(m)\tau), b+\frac{e^{kb}}{k}\cos(\tau)\right), \qquad\tau\geq0,
\end{equation}
is a parametrization of the path traced  by the particle.
The case $c=m$ corresponds to Gerstner's wave analyzed in   \cite{Co01, He08, AM12xx}, when the particles move on circles.
When $U\neq0,$ the curve parametrized by $(x,z)$ still has a nice geometric interpretation. 
Namely, if  $U\sign(m)>0$, then it represents a  trochoid 
which is locus of a point at distance $e^{kb}/k$ from the center of a circle of radius  $|U|/(k|m|)$
rolling on a fixed line. 
There are three special cases which may occur:
 the point may be located inside the circle and in this case the curve is  real-analytic and is called curtate cycloid;
 the point may be located on the circle and the curve has downwards cusps at $\tau=(2l+1)\pi/2$, $l\in\Z,$ being called cycloid; or  the point is located outside the circle and  the curve has self-intersections and is called prolate cycloid.
On the other hand, if  $U\sign(m)<0$,   the curve traced  by the particle is a reflected trochoid (the refection of a trochoid  with respect to an horizontal line).

In the remainder, we  shall consider the most general case when the parameter $b_0$ in \eqref{regime} is $b_0=0.$

\subsection{Particle paths in the absence of the Coriolis force}

We first describe the fluid motion underlying the waves  given by \eqref{Lag} when $\omega=0.$
In this regime, the motion is completely determined by the underlying current $U$ and the sign of  $m$.\footnote{Given $U\in\R,$ the wave speed $c$ and the value of $m$ are given by the relations
\[m=\sign(m)\sqrt{\frac{g}{k}}\qquad\text{and}\qquad c=U+\sign(m)\sqrt{\frac{g}{k}}.\]}
Furthermore, by Lemma \ref{L:2a}, we may choose $U$ to be any  real number since there is no restriction on the wave speed $c$.    
\begin{figure}
$$\includegraphics[width=12cm, angle =0]{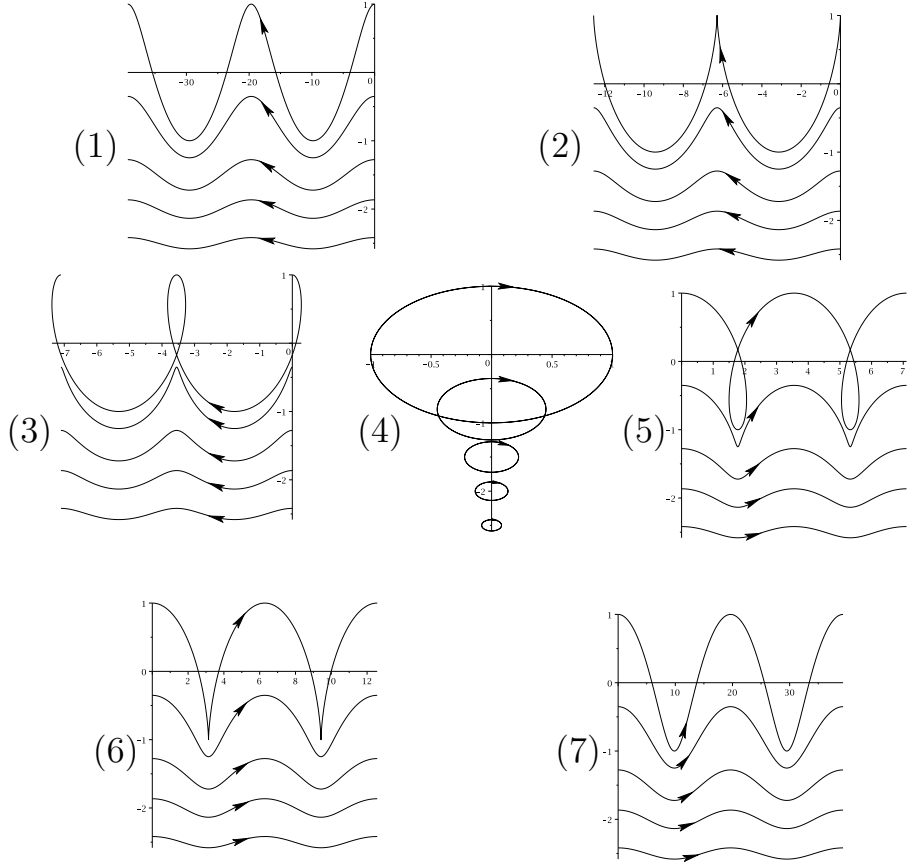}$$
\caption{Particle paths in the absence of the Coriolis force when $k=1$, $m=\sqrt{g},$ and the current $U$ is given by: $(1)$ $ U=-g$; $(2)$ $ U=-\sqrt{g}$; $(3)$ $ U=-g^{1/4}$; $(4)$ $  U=0$; $(5)$ 
$ U=g^{1/4}$; $(6) $ $ U=\sqrt{g}$; $(7)$ $ U=g$. 
In each of the figures $(1)-(7)$ we plotted the trajectories of the particle parametrized by $(a,b)$ over two periods when $a=0$ and $b=0,$ $b=-0.8,$ $b= -1.5,$ $b=-2,$ and $b=-2.5.$
Observe that as $b$ decreases, the vertical displacement in each half-period decreases too.}
\label{F:1}
\end{figure}

\begin{prop}  \label{o=0}
Assume that $\omega=0,$ and let  $U$ be the  velocity of the   horizontal background current.
Furthermore, we assume that   $m=\sqrt{g/k}$ and consider a  particle located in the fluid layer and parametrized by  $(a,b)\in\Sigma$. 
Then, we have:
\begin{itemize}
\item[(a)] In the absence of a  current $U=0,$ the  particle moves clockwise on a circle.
 \item[(b)] If $U>0,$ (resp. $U<0$) then the particle moves on a trochoid (resp. the reflection  of a  trochoid with respect to a vertical line).  
Letting $b_*:=\ln ((kU^2/g)^{1/2k}),$
the   trochoid is of the one of the following types:  prolate cycloid if $b>b_*;$  a cycloid  if $b=b_*$; or  a curtate cycloid if $b<b_*.$ 
 \end{itemize}
If $m=-\sqrt{g/k}$, the situation is similar:  the  particle moves counterclockwise on a circle if $U=0$,
and, if $U<0$ (resp. $U>0$), then it moves on a  trochoid (resp. a reflected   trochoid). 

The particle paths are illustrated in Figure \ref{F:1}.
\end{prop}
\begin{proof}
 The   proof is a consequence of  \eqref{param} and Lemma \ref{L:2a}. 
\end{proof}

Note that the constant $b_*$  is non-positive, which is the only relevant situation for Proposition \ref{o=0},  if and only if
\begin{equation}\label{eq:a}
 0<|U|\leq\sqrt{g/k}.
\end{equation}
Assuming  \eqref{eq:a}, particles close to the wave surface are looping in the direction of the current, while those located deeper in the fluid undulate in the same direction. 
On the other hand, if  \eqref{eq:a} does not hold and $U \not= 0$, then all the particle undulate in the direction of  the current. 
For $U=-\sign(m)\sqrt{g/k} $ we have  $c=0$ and  $b_*=b_0=0.$ 
When $\sign(m) U\in(0,\sqrt{g/k})$ the current is moving in the opposite direction to the wave (adverse current), otherwise (excepting $U\in\{0,-m\}$)
the current moves in the same direction as the wave.

\subsection{Particle paths for geophysical waves}
We  describe now the paths traced by the particles for the wave solutions \eqref{Lag} in the geophysical context when $\omega>0.$
We  emphasize that for geophysical waves there exists a upper bound on the value of $U$ above which the equations \eqref{Lag} cease to describe solutions of the problem \eqref{Pb1}:
\begin{equation}\label{restr1}
 U\leq \frac{\omega^2+g}{2\omega}.
\end{equation}
  Indeed,  by Lemma \ref{L:2b}, we can determine $U$ in dependence  of $m$ 
\begin{equation*}\label{U}
U=\frac{g-km^2-2\omega m}{2\omega},
\end{equation*}
and this quadratic equation in $m$ has real solutions if and only if \eqref{restr1} is satisfied.
Furthermore, there is also an upper bound on the speed at which the wave may travel eastwards:
\begin{equation}\label{restr2}
 c\leq \frac{ g}{2\omega}.
\end{equation}
While in the case $\omega=0$, the fluid layer associated to the solutions \eqref{Lag} does not contain stagnation points, 
that is particles which travel horizontally with the wave speed, in the geophysical context the situation is different.

\begin{prop}\label{P:1} Assume that $m=0$. 
Then, all the  particles  corresponding to the solution \eqref{Lag} move horizontally eastwards with the wave speed $c=U=g/(2\omega).$
If  $m\neq0$, then the equations \eqref{Lag} describe waves without stagnation points.
\end{prop}
\begin{proof} It follows readily from \eqref{Lag} that
\[
(u,w)=(c-m(1-e^{kb}\sin(k(a-mt))), me^{kb}\cos(k(a-mt))). 
\]   
Clearly, if $m=0$, then $(u,w)=(c,0)$  and all the fluid particles are stagnation points. 
If $m\neq0$, there is no pair $(a,b)$ with this property.
The assertion follows now from Lemma \ref{L:2b}.
\end{proof}

\begin{figure}
$$\includegraphics[width=14cm, angle =0]{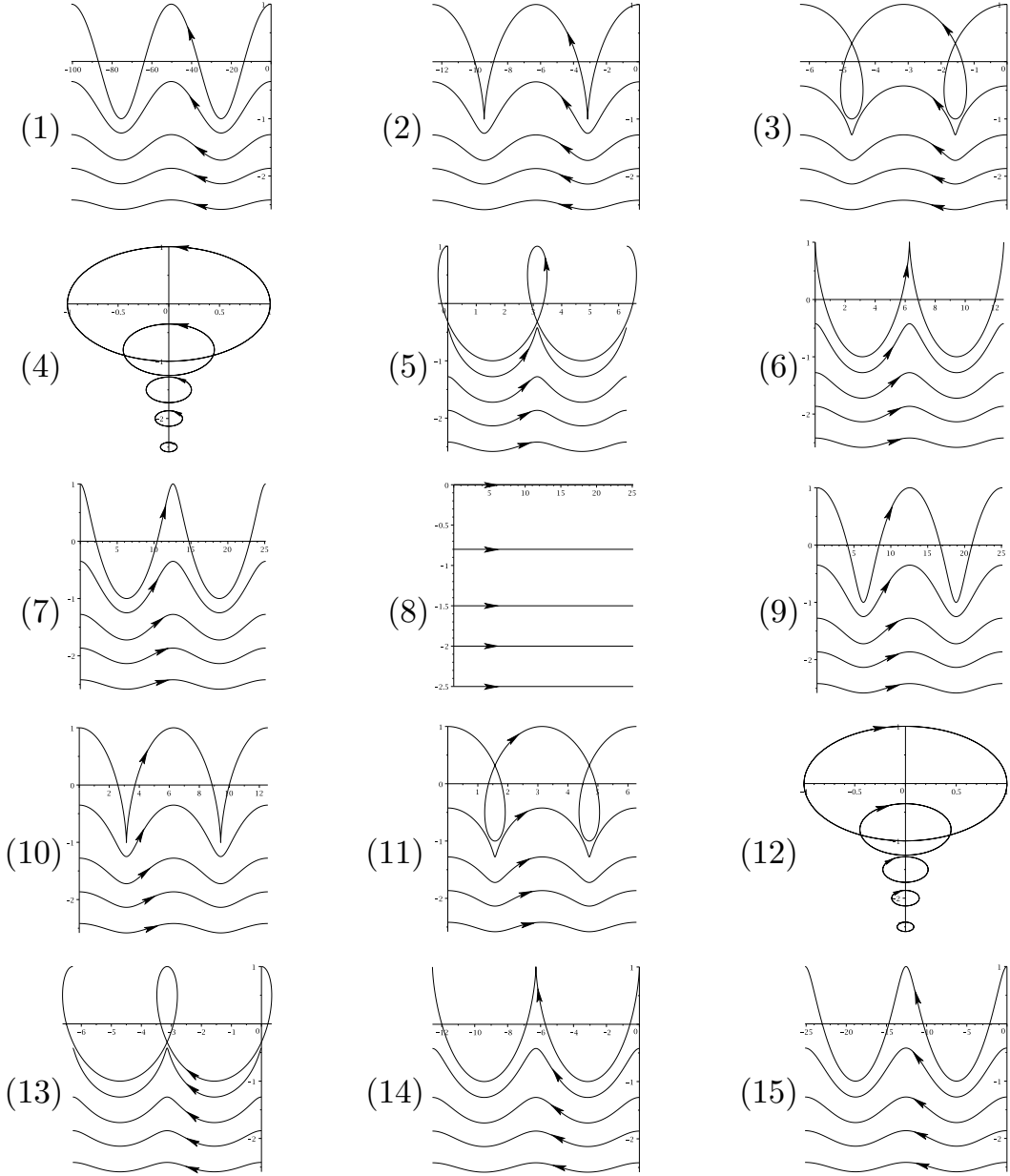}$$
\caption{Particle paths for geophysical waves with $k=1$ when: {\small $(1)$ $m=-9\omega-\sqrt{g}$; $(2)$ $m=-2\omega-\sqrt{4\omega^2+g}$; $(3)$ $m=-3/2\omega-\sqrt{\omega^2+g}$;
$(4)$ $m=-\omega-\sqrt{\omega^2+g}$; $(5)$ $m=-\omega/2-\sqrt{\omega^2+g}$; $(6)$ $m=-\sqrt{g}$; $(7)$ $m=-\omega-\sqrt{g}$;
$(8)$ $m=0$; $(9)$ $m=-3\omega+\sqrt{4\omega^2+g}$; $(10)$ $m=-2\omega+\sqrt{4\omega^2+g}$; $(11)$ $m=-3/2\omega+\sqrt{(3\omega/2)^2+g}$;
$(12) $ $m=-\omega+\sqrt{\omega^2+g}$; $(13)$ $m=-\omega/2+\sqrt{(\omega/2)^2+g}$; $(14)$ $m=\sqrt{g}$; $(15)$ $m=\omega+\sqrt{\omega^2+g}$.}
Each figure  $(1)-(15)$ shows the paths of the particle   $(a,b)$ over two periods, with $(a,b)$ as in Figure \ref{F:1}.}
\label{F:2}
\end{figure}

Since for equatorial geophysical waves $U$ and $c$ are determined by the constant $m,$ it is natural to use $m$ as a parameter when studying the motion of the particles for the solutions \eqref{Lag}.  
We obtain the following  characterization of the paths described by the particles. 
\begin{prop}  \label{o}
Assume that $\omega>0$, let    $(a,b)\in\Sigma$ be the parameter corresponding to a particle in the fluid, and define the constants
\[m_1:=\frac{-\omega- \sqrt{\omega^2+gk}}{k}<m_2:=\frac{-\omega+ \sqrt{\omega^2+gk}}{k}.\] 
Then,  we have:
\begin{itemize}
\item[(a)] If $m=m_1$ (resp. $m=m_2$), then  the  particle moves counterclockwise (resp. clockwise) on a circle.
 \item[(b)] If $m\in(-\infty,m_1)\cup (0,m_2),$ then the particle moves on a   trochoid.  
Letting 
\[b_*:=\frac{1}{k}\ln \left(\frac{g-km^2-2\omega m}{2\omega m}\right)\]
the   trochoid is of one of the following types:  prolate cycloid if $b>b_*;$  a cycloid  if $b=b_*$; or  a curtate cycloid if $b<b_*.$ 
\item[(c)] If $m\in(m_1,0)\cup (m_2,\infty),$ then the particle moves on a reflected trochoid.  
Letting 
\[b^*:=\frac{1}{k}\ln \left(\frac{g-km^2-2\omega m}{-2\omega m}\right)\]
the   trochoid is of one of the following types:  prolate cycloid if $b>b_*;$  a cycloid  if $b=b_*$; or  a curtate cycloid if $b<b_*.$ 
 \end{itemize}

The particle paths are illustrated in Figure \ref{F:2}.
\end{prop}
\begin{proof}
 The proof follows from the Lemma \ref{L:2b} and relations \eqref{param}.
\end{proof}

 The relevant case $b_*\leq0$ occurs in the case $(b)$ if and only if $m\in[m_3,m_1)\cup [m_4,m_2)$, and in the case $(c)$ we have $b^ *\leq0$  if and only if $m\in(m_1,-\sqrt{g/k}]\cup(m_2,\sqrt{g/k}],$
whereby $m_{3/4} $ are given by
\[
m_3:=\frac{-2\omega- \sqrt{4\omega^2+gk}}{k}\qquad\text{and}\qquad m_4:=\frac{-2\omega+ \sqrt{4\omega^2+gk}}{k}.
\]
The current is favorable when $m\in (-\infty,m_1)\cup(-\sqrt{g/k},m_2)\cup(\sqrt{g/k},\infty) $ and we have a  adverse current when $m\in(m_1,-\sqrt{g/k})\cup(m_2,\sqrt{g/k}).$
 If $ (m,c,U) =(\sqrt{g/k},0,-\sqrt{g/k})$ or $(m,c,U)=(-\sqrt{g/k},0,\sqrt{g/k})$, then the equations \eqref{Lag} describe solutions in both cases $m=0$ and $m>0$.
In fact, these are the only solutions in the family \eqref{Lag} which are valid in both contexts (see Figure \ref{F:1} and Figure \ref{F:2} for the paths of the particles in these cases).

\vspace{0.3cm}
\hspace{-0.5cm}{\large \bf Acknowledgment} 
This research has been supported by the FWF Project I544 --N13 ``Lagrangian kinematics of water waves'' of the Austrian Science Fund.\\

\bibliographystyle{abbrv}
\bibliography{AM.bib}
\end{document}